\documentclass[11pt]{article}

\usepackage{cite}
\usepackage{lineno}
\usepackage{mathrsfs} 
\usepackage{ifthen}

\usepackage{exscale}
\usepackage{tabularx}
\usepackage{latexsym}
\usepackage{amsmath,amssymb,amstext,amsthm} 
\usepackage{graphicx,color,epsfig}
\usepackage{graphicx}
\usepackage{fullpage}        
\usepackage{float}
\usepackage{verbatim}
\usepackage[bookmarks,pagebackref,
    pdfpagelabels=true, 
    ]{hyperref}
\usepackage{makeidx}
\makeindex

\numberwithin{equation}{section}

\def\R{\mathbb{R}}
\def\S{\mathbb{S}}

\def\Rn{\mathbb{R}^n}

\def\eqref#1{{\normalfont(\ref{#1})}}

\newcommand{\snr}[1]{\mathcal{S}_{#1+}^n}
\newcommand{\psnr}[1]{\mathcal{P}(\mathcal{S}_{#1+}^n)}

\def\eqref#1{{\normalfont(\ref{#1})}}

\newtheorem{assump}{Assumption}[section]
\newtheorem{prop}{Proposition}[section]

\newtheorem{ques}{Question}[section]

\newtheorem{theorem}{Theorem}[section]
\newtheorem{cor}{Corollary}[section]

\newtheorem{remark}{Remark}[section]

\newtheorem{lemma}{Lemma}[section]

\newcommand{\textdef}[1]{\textit{#1}\index{#1}}

\newcommand{\range}{\mathrm{range}}

\newcommand{\LL}{{\mathcal L} }

\newcommand{\PP}{{\mathcal P} }

\newcommand{\Sn}{{\mathcal S^n}}

\newcommand{\Snp}{{\mathcal S^n_+\,}}

\newcommand{\Snrp}{{\S^n_{r+}\,}}
\newcommand{\Snrpc}{{\mathcal S^n_{r+}}}
\newcommand{\Snirpc}{{\mathcal S^{n_i}_{r+}}}

\newcommand{\A}{{\mathcal A}}

\newcommand{\II}{{\mathcal I\,}}

\newcommand{\bbm}{\begin{bmatrix}}
\newcommand{\ebm}{\end{bmatrix}}
\newcommand{\bem}{\begin{pmatrix}}
\newcommand{\eem}{\end{pmatrix}}
\newcommand{\beq}{\begin{linenomath*} \begin{equation}}
\newcommand{\beqs}{\begin{linenomath*} \begin{equation*}}
\newcommand{\eeq}{\end{equation} \end{linenomath*}}
\newcommand{\eeqs}{\end{equation*} \end{linenomath*}}
\newcommand{\beqr}{\begin{linenomath*} \begin{eqnarray}}
\newcommand{\beqrs}{\begin{linenomath*} \begin{eqnarray*}}
\newcommand{\eeqr}{\end{eqnarray} \end{linenomath*}}
\newcommand{\eeqrs}{\end{eqnarray*} \end{linenomath*}}
\newcommand{\bet}{\begin{table}}

\DeclareMathOperator{\trace}{{trace}}
\DeclareMathOperator{\codim}{{codim}}

\DeclareMathOperator{\rank}{{rank}}

\newcommand{\nc}{\newcommand}
\nc{\arrow}{{\rm arrow\,}}
\nc{\Arrow}{{\rm Arrow\,}}
\nc{\BoDiag}{{\rm B^0Diag\,}}
\nc{\bodiag}{{\rm b^0diag\,}}

\nc{\Mm}{{\mathcal M}^{m} }
\nc{\Mmn}{{\mathcal M}^{mn} }
\nc{\Mnr}{{\mathcal M}_{nr} }
\nc{\Mnmr}{{\mathcal M}_{(n-1)r} }
\nc{\kwqqp}{Q{$^2$}P\,}
\nc{\kwqqps}{Q{$^2$}Ps}

\nc{\notinaho}{(X,S)\in \overline{AHO}(\A)}
\nc{\inaho}{(X,S)\in AHO(\A)}

\newcommand{\bea}{\begin{eqnarray}}%
\newcommand{\eea}{\end{eqnarray}}%
\newcommand{\beas}{\begin{eqnarray*}}%
\newcommand{\eeas}{\end{eqnarray*}}%
\newcommand{\Rmn}{\R^{m \times n}}%
\renewcommand{\L}{\mathcal{L}}%
{}

\newcommand{\Hnp}[1][]{\,\mathbb{H}_+^{\ifthenelse{\equal{#1}{}}{n}{#1}}}
\newcommand{\Hn}[1][]{\,\mathbb{H}^{\ifthenelse{\equal{#1}{}}{n}{#1}}}
\newcommand{\Dn}[1][]{\,\mathbb{D}^{\ifthenelse{\equal{#1}{}}{n}{#1}}}

\begin{document}

	\bibliographystyle{plain}
\title{
	Rank Restricted Semidefinite Matrices\\
	and\\
	Image Closedness\footnote{Research supported by
		The Natural Sciences and Engineering Research Council of Canada.}
}
             \author{
		     \href{https://www.linkedin.com/pub/ian-davidson/87/2bb/6a5}
{Ian Davidson}\thanks{
Department of Combinatorics and Optimization, University of Waterloo, Waterloo, Ontario, Canada N2L 3G1. 
	}
\and
\href{http://www.math.uwaterloo.ca/~hwolkowi/}
{Henry Wolkowicz}%
        \thanks{Department of Combinatorics and Optimization
Faculty of Mathematics, University of Waterloo, Waterloo, 
Ontario, Canada N2L 3G1, \url{www.math.uwaterloo.ca/\~hwolkowi}.
}
}
\date{\today}
          \maketitle



\begin{abstract}
We study the closure of the projection of the (nonconvex) cone of
rank restricted positive semidefinite matrices onto subsets of 
the matrix entries. This defines the feasible sets for
semidefinite completion problems with restrictions on the
ranks. Applications include conditions
for low-rank completions using the nuclear norm heuristic.
\end{abstract}



\tableofcontents

\section{Introduction}

Consider an \textdef{undirected graph, $G=(V,E)$}, 
\index{$G=(V,E)$, undirected graph}
with \textdef{vertex set, $V=\{1,\ldots,n\}$}, 
\index{$V=\{1,\ldots,n\}$, vertex set}
and \textdef{index set, $E\subseteq \{ij:i\leq j\}$}.
\index{$E\subseteq \{ij:i\leq j\}$, index set} 
The classical positive semidefinite (PSD) completion problem begins with
a given \textdef{partial symmetric matrix} $X\in \Sn$, where
$X_{ij}=a_{ij}, \forall ij \in E$ and attempts to find the missing
entries from the data $a\in \R^E$ 
so that $X$ is PSD. One of the problems in \cite{DrPaWo:14} answers
the question of when the projection of the PSD cone $\Snp$ onto the
matrix entries indexed by $E$ is closed, i.e.,~when the set of
\textdef{coordinate shadows, $\PP(\Snp)$}, is closed.
In this paper we add an additional rank restriction and ask the
following.
\begin{ques} When is the projection of the (generally nonconvex)
cone of PSD matrices of rank at most $r$, $\PP(\Snrpc)$, closed?
\end{ques}
\index{PSD matrices of rank at most $r$, $\Snrpc$}
\index{$\Snrpc$, PSD matrices of rank at most $r$}
Such questions arise for example in constraint qualifications for
guaranteeing strong duality. It is also closely related to the
closedness of the sum of sets. See
e.g.,~\cite{DrPaWo:14,BorweinMoorsB:10,Pataki:07,MR1404832}.
In addition, the projection $\PP(\Snrpc)$ is exactly 
the data set determining the feasibility of the rank restricted 
PSD completion problem, e.g.,~\cite{MR90g:15039,MR2002b:15002}.

The paper is organized as follows. We continue with the background and
some preliminary results in the remainder of this section. We then
show that we can restrict our attention to
the connected components of the graphs in Section \ref{sect:partG}.
Specific cases for closure and failure of closure are given in Section
\ref{sect:bipartG}. In particular we show the importance of bipartite
graphs. The concluding Section \ref{sect:concl} contains a summary of
the results and some open questions and conjectures.

\subsection{Background}
We work in the space of $n\times n, n\geq 2$, 
\textdef{real symmetric matrices, $\Sn$},
\index{$\Sn$, real symmetric matrices}
equipped with the \textdef{trace inner product} $\langle A,B \rangle = 
\trace AB, \forall A,B\in \Sn$. We denote the closed convex cone of 
\textdef{PSD matrices, $\Snp$}. We focus on the
\index{$\Snp$, PSD matrices}
generally nonconvex cone of PSD matrices of rank at most $r$, $\Snrpc$.
We allow for self-loops in the undirected graph, $G=(V,E)$, and denote
\textdef{$\LL = \{i:ii \in E\}$} with the complement \textdef{$\LL^c$}.
A partial symmetric matrix $X\in \Sn$ is is called
a \textdef{partial PSD matrix} if
all the principal submatrices formed by known entries 
are PSD. The \emph{PSD matrix completion problem with rank restriction 
at most $r$} can be stated as completing the partial PSD matrix $X$ to a
positive semidefinite matrix such that the rank of the 
completion is at most $r$. We assume that $1\leq r\leq n$.

Our work extends the following.
\begin{theorem}[{\cite{DrPaWo:14}}] 
	\label{thm:origmain}
$\PP(S_+^n)$ is closed if and only if $\L,\L^c$ are disconnected.
\qed
\end{theorem}
\subsection{Preliminary Results}
We first add the following related result.
\begin{theorem}[{\cite[Thm 1.1]{MR1797294},\cite{Gpat:95},\cite[Sect.
31.5]{MR98g:52001}}]
\label{thm:barv}
Let $\A\subset \Sn$ be an affine subspace such that the intersection
$\A\cap \Snp$ is non-empty and $\codim(\A) \leq \begin{pmatrix}r+2\cr 2
\end{pmatrix} -1$ for some non-negative integer $r$. Then there is a
matrix $X\in \A\cap \Snp$ such that $\rank(X)\leq r$.
\qed
\end{theorem}
We now note the following two results that follow from the above two
theorems.
\begin{cor}
\label{cor:nm1}
Let 
\begin{equation}
\label{eq:tE}
t =  \left\lceil {-\frac 32 + 
               \frac {\sqrt {9 + 8|E|}}2} \right\rceil.
\end{equation}
Then $t\leq n-1$ and 
\[
\Big(  \psnr{r}  \text{ is closed for } r=t,t+1,\ldots,n  \Big)   
\iff
 \Big(\L \text{ and } \L^c \text{ are disconnected}\Big).
\]
\end{cor}
\begin{proof}
Necessity follows from the Theorem \ref{thm:origmain} if we fix $r=n$.

For sufficiency first note that Theorem \ref{thm:origmain} implies
$\PP(\Snp)$ is closed. Moreover, the closure holds if $|\LL|=n$, since
any sequence of partial PSD matrices $a^i\to a$ with $\PP(X^i)=a^i, X^i\in
\psnr{n-1}$ means that the diagonal elements converge and so we can
assume that $X^i\to \bar X\succeq 0$. The rank result now follows from
lower semi-continuity of the rank function. Therefore, we can assume
that $\LL^c\neq \emptyset$.

Now suppose that $r=n-1$ and consider a sequence of
partial PSD matrices $a^i\to a$ and suppose that $\PP(X^i)=a^i, X^i\in
\psnr{n-1}$. Then there exists $\bar X\succeq 0, \PP(\bar X)=a$. If
$\rank(\bar X)\leq
n-1$ then we are done. If $\rank(\bar X)=n$, then we can consider the
positive semidefinite completion problem, (PSDC),
\index{positive semidefinite completion problem, PSDC}
\index{PSDC, positive semidefinite completion problem}
\begin{equation}
\label{eq:psdcrank}
(PSDC)\qquad 
\min \trace(CX) \text{ s.t. } X\succeq 0, \, X_{ij}=a_{ij}, \, \forall ij \in E.
\end{equation}
This program has a feasible solution $\bar X \succ 0$ that is not unique
since we have $\LL^c \neq \emptyset$. Therefore we can move in the
direction $\bar X + \alpha D, \alpha \in \R$, for some $0\neq D\in \Sn$.
This means that $\bar X + \alpha D \notin \Snp$, for some $\alpha \in
\R$, and on the line segment $[\bar X, \bar X+\alpha D]$
we can find a singular feasible point
$\bar X + \bar \alpha D\succeq 0$, for some $\bar \alpha \in\R$. The
closure follows since feasibility means
$\PP(\bar X + \bar \alpha D)=a$.

The key to the sufficiency proof above was in finding a feasible
SDP completion with the lower rank. We do this by applying Theorem
\ref{thm:barv}. The codimension for a completion problem is exactly
$|E|$, the number of constraints or elements that are fixed.
Therefore, we have $|E| \leq \frac {(t+2)(t+1)}2-1$ which is equivalent
to $2|E|+2 \leq t^2+3t+2$. The only non-negative root for this quadratic
yields the smallest nonnegative integer
\[
t= \left\lceil {-\frac 32 + \frac {\sqrt {9 + 8|E|}}2} \right\rceil.
\]
We can combine this with the result for $n-1$ and obtain a feasible solution 
$X$, a completion, with rank at most $t$, i.e.,
\[
X\in \psnr{t}, \quad \PP(X)=a.
\]
\end{proof}
\begin{cor}
\label{cor:trankbnd}
Suppose that $|E|< \frac 12 (t^2+3t) \,\left(=\begin{pmatrix}t+2\cr
2\end{pmatrix}-1\right),\, t<n$. Then
\[
\Big(  \psnr{r}  \text{ is closed for } r=t,t+1,\ldots,n  \Big)   
\iff
 \Big(\L \text{ and } \L^c \text{ are disconnected}\Big).
\]
\end{cor}
\begin{proof}
We just square both sides in \eqref{eq:tE}.
\end{proof}
\begin{cor}
\label{cor:discon}
Let $\LL$ and $\LL^c$ be connected. Then $\psnr{r}$ is not closed.
\end{cor}
\begin{proof}
The proof is similar to the general case in \cite{DrPaWo:14}. 
We include it for completeness.
Without loss of generality, we can assume that $1\in \LL$, $2\in \LL^c$ and 
$12\in E$. Taking a sequence of partial matrices $a^i$ with 
$a^i_{11} = \frac{1}{i}$, $a^i_{12} = 1$ and all other entries of $a^i = 0$. 
Then we have a sequence of matrices and images
\beqs 
X^i = \begin{bmatrix}
\frac{1}{i} & 1 & 0 & \ldots\\
1 & ? & 0 & \ldots\\
0 & 0 & 0 & \ldots\\
\vdots & \vdots & \vdots & \ddots\\
\end{bmatrix} \in \Snrpc, \quad
	\PP(X^i)=a^i.
\eeqs
This sequence of matrices is always rank one completable with
$X_{22}=i$. (And thus is rank at most $r$ completable.)
Therefore $a^i \in \psnr{r}$. But $a^i \underset{i \rightarrow
\infty}\longrightarrow \bar{a}$ with $\bar{a}_{11} = 0$. 
This is not PSD completable.
\end{proof}

Following Corollaries \ref{cor:nm1} and \ref{cor:discon} we can add the 
following.

\begin{assump}
	\label{assump:disconn}
In the remainder of this paper we assume that
$\LL$ and $\LL^c$ are disconnected in the undirected graph $G=(V,E)$ and
\[1\leq r
 <  \left\lceil {-\frac 32 + 
               \frac {\sqrt {9 + 8|E|}}2} \right\rceil
\, \left(\leq (n-1) \right).
\]
\end{assump}
\index{connected components in $G$}
\index{disconnected components in $G$}

\section{Partitioned Graphs}
\label{sect:partG}
\index{partitioned graph}
We now get a rather nice result for closure that allows us to assume, without
loss of generality, that our graphs are composed of 
\emph{two connected components}. As an illustration, 
we first show the following.
\begin{prop}
	\label{prop:k2}
Suppose $X=\begin{bmatrix} A & ?\\ ? & B\\ \end{bmatrix}$ is a partial
PSD matrix, i.e.,~both $A$ and $B$ are PSD matrices.
Then the minimum rank PSD completion of $X$, denoted $\bar X$, has 
	\[
\rank(\bar X)= \max\{\rank(A),\rank(B)\}.
	\]
Moreover, the maximum rank PSD completion has rank given 
by the sum, $\rank(A)+\rank(B)$.
\end{prop}

\begin{proof}
We use the unique PSD square roots of $A,B$ and get the completion with
the correct rank
\[
	\bar X= 
	\begin{bmatrix}
		A^{1/2} \\ B^{1/2}
	        \end{bmatrix}
	\begin{bmatrix}
		A^{1/2} \\ B^{1/2}
	        \end{bmatrix}^T\succeq 0,
\]
i.e.,we get $\rank(\bar X)= \rank \left(
	\begin{bmatrix}
		A^{1/2} \\ B^{1/2}
	\end{bmatrix}\right)$.

The maximum rank completion is obtained by using zeros in the
off-diagonal blocks.
\end{proof}

\begin{theorem}
\label{thm:partitclose}
Let $\{H_i\}_{i=1}^k$ be a partition of $V$,  
\[
H_1, \ldots, H_k \subseteq V, \, \cup_{i=1}^k H_i=V, \,
H_i \cap H_j =\emptyset, \forall i\neq j, \quad
n_i:= |H_i|, \,i=1,\ldots,k.
\]
Then the projection $\PP (\Snrpc)$ is closed if, and only if, 
the restricted projections to each component
$\PP_{H_i}(\Snirpc)$ is closed for all $i=1,\ldots,k$.
\end{theorem}
\begin{proof}
Necessity follows by considering the case of setting all the elements 
in all the components but one to zeros.

For sufficiency we consider the sequence with convergent projections
\[X^j= \begin{bmatrix}
	X_{11}^j & \ldots  & X_{1k}^j \\
	\vdots & \ddots & \vdots \\
	(X_{1k}^j)^T & \ldots & X_{kk}^j
	\end{bmatrix}  \in \Snrpc, \qquad x^j=\PP(X^j) \to x \in \R^E,
\quad j=1,2,\ldots.
\]
We need to find $X\in \Snrpc$ with $x=\PP(X)$.
Denote the restricted projections
\[
	x_i^j:= \PP_{H_i}(X_{ii}^j)\to x_i, \, i=1,\ldots k.
\]
From the closure condition, we can now conclude that there exist
$X_i\in \Snirpc$ with $\PP_{H_i}(X_{i}) =x_i, \forall i$.
We can now obtain the desired completion with appropriate rank by using
\[
	X= 
	\begin{bmatrix}
		X_1^{1/2}   \\
		X_2^{1/2}   \\
		\ldots        \\
		X_k^{1/2}
	\end{bmatrix} 
	\begin{bmatrix}
		X_1^{1/2}   \\
		X_2^{1/2}   \\
		\ldots        \\
		X_k^{1/2}
	\end{bmatrix}^T,
\]
i.e.,~we apply the idea from Proposition \ref{prop:k2}.
\end{proof}

\begin{cor}
\label{cor:connected}
The projection $\PP(\Snrp)$ is closed if, and only if, the restricted
projections $\PP_{H_i}(\Snirpc)$ are closed for all
\underline{connected} components
$H_i$ of $G$, $n_i=|H_i|$.
\end{cor}
\begin{proof}
From Theorem \ref{thm:partitclose} we can restrict to components.
From Assumption \ref{assump:disconn} we can restrict to connected
components.
\end{proof}
We can now focus on the connected components of a graph; equivalently, 
we can deal with each component separately and so
assume we are dealing with a connected graph.
\begin{assump}
\label{assump:connected}
Assume that Assumption \ref{assump:disconn} holds and that
the graph $G$ is connected with
\[
|\L|=n, \text{   or   }   |\L|=0.
\]
\end{assump}

\subsection{Closure for Loop Graphs, $|\L|=n$}
This result follows from a similar proof to the main result in
\cite{DrPaWo:14} or as a corollary to Theorem \ref{thm:origmain}.
\begin{theorem}
Let $|\LL| =  n$. Then $\PP(\Snrpc)$ is closed.
\end{theorem}
\begin{proof}
Suppose we have a sequence of matrices, $\{X^j\} \subset \Snrpc$ with
$\PP(X^j) = x^j \to x$. Then the diagonal elements of $X^j$ converge and
therefore the off-diagonal elements are bounded. Therefore, without loss
of generality $X^j\to X$. The result now follows from the closure of
$\Snp$ and the lower semi-continuity of rank.
\end{proof}

\subsection{Examples of Failure for Loopless Graphs, $|\LL| = 0$}
We note that the following follows from the above results.
We include a proof since it emphasizes the elementary nature for $r=n$
and the difficulty that might arise for $r<n$.
\begin{theorem}
Let $r \in \{0,n\}, \, |\LL| =  0$. Then $\PP(\Snrpc)$ is closed.
\end{theorem}
\begin{proof}
The $r=0$ follows from $\PP(0)=0$. For $r=n$,
we can always set the unspecified off-diagonal elements to $0$; and then
we set the diagonal elements large enough to ensure positive definiteness.
\end{proof}
From our results we now only have one case to consider: $0<r<n$ and 
all the vertices of our connected graph are loopless. Unfortunately,
we do not have simple results for closedness.

We will look at exclusions to begin with, and
then provide theorems for closure and completion.

\subsubsection{Rank One Case, $r=1$}
\label{rankoneexclusions}
We begin by looking at examples of the simplest case, the rank one case, $r=1$.
In fact, the following two instances characterize failure of closure for
the rank one case, see Corollary \ref{cor:rankone}, below.
\begin{lemma}
\label{lem:triangle}
If the graph $G$ has a triangle, a cycle of length $3$,
then $\psnr{1}$ is not closed.
\end{lemma}
\begin{proof}
Without loss of generality, we can let the triangle be formed by the 
vertices $\{1,2,3\}$. Let 
\[
v^j = \begin{pmatrix} \frac{1}{\sqrt j} & \frac{1}{\sqrt j} & \sqrt
j & 0 &  \ldots & 0 \end{pmatrix}^T
\in \Rn, \quad X^j = v^j(v^j)^T \in \snr{1}. 
\]
Then, with $E=\{12,13,23, \ldots\}$, we have
\[
\PP(X^j)=
 \begin{pmatrix} \frac{1}{j} & 1  &1 &  0 &
\ldots & 0 \end{pmatrix}^T \to \begin{pmatrix} 0 & 1  &1 &  0 &
\ldots & 0 \end{pmatrix}^T,
\]
and
$\begin{bmatrix}
? & 0 & 1 & \ldots\\
0 & ? & 1 & \ldots\\
1 & 1 & ? & \ldots\\
\vdots & \vdots & \vdots & \ddots
\end{bmatrix}$ has no rank one completion.
\end{proof}
\begin{lemma}
\label{lem:path}
If $G$ has a path of length $3$ that is not a cycle (of length
$4$), then $\psnr{1}$ is not closed.
\end{lemma}
\begin{proof}
Without loss of generality, we can let the path be defined by
the first four distinct vertices $1,2,3,4$. Let
\[
v^j = \begin{pmatrix} \sqrt j&\frac{1}{\sqrt j} &  \frac{1}{\sqrt j} 
& \sqrt j & 0 &  \ldots & 0 \end{pmatrix}^T
\in \Rn, \quad X^j = v^j(v^j)^T \in \snr{1}. 
\]
Then, with $E=\{12,23,34 \ldots\}$\footnote{We could choose
$E=\{12,13,23,24,34 \ldots\}$, $E=\{12,23,24,34 \ldots\}$,
or $E=\{12,13,23,34 \ldots\}$.}
\[
\PP(X^j)=
 \begin{pmatrix} 1 & \frac{1}{j} & 1  &0 &  0 &
\ldots & 0 \end{pmatrix}^T \to \begin{pmatrix} 1 & 0  &1 &  0 &
\ldots & 0 \end{pmatrix}^T,
\]
and
  $ \begin{bmatrix}
? & 1 & ? & ? & \ldots\\
1 & ? & 0 & ? & \ldots\\
? & 0 & ? & 1 & \ldots\\
? & ? & 1 & ? & \ldots\\
\vdots & \vdots & \vdots & \ddots
\end{bmatrix}$ has no rank one completion.
\end{proof}
Note that the path in the instance in the proof of Lemma \ref{lem:path} 
cannot be a cycle since the $(1,4)$ entry in $X^j$ diverges to $+\infty$.
\begin{remark}
Note that we could extend the above two lemmas to higher rank using
orthogonal vectors. For example, for Lemma \ref{lem:triangle} with rank 
$3$, we could use a cycle of length $5$ and use two orthogonal vectors
\[
v_{\pm}^j = \begin{pmatrix} 
\frac{1}{\sqrt j} & \pm \frac{1}{\sqrt j} &
\frac{1}{\sqrt j} & \pm \frac{1}{\sqrt j} 
& \sqrt j &  \pm \sqrt j 
& 0 &  \ldots & 0 \end{pmatrix}^T
\in \Rn, \quad X^j = \sum_{\pm} v_{\pm}^j(v_{\pm}^j)^T \in \snr{2}. 
\]
The limit yields a partial matrix with no rank $2$ PSD completion.
We could similarly extend Lemma \ref{lem:path}.
\end{remark}

\section{Bipartite Graphs, Independent Sets, Cliques}
\label{sect:bipartG}
We now look at first sufficient and then necessary conditions for closure.
Recall that Assumption \ref{assump:connected} holds.

\subsection{Complete Bipartite Graphs}
The graphs in the above two examples in Lemmas \ref{lem:triangle} and 
\ref{lem:path} where closure can fail for $\psnr{1}$
are both complete bipartite graphs. Recall that a graph is 
\textdef{bipartite} means that it is \emph{$2$-colourable}, 
i.e.,~the vertices can be 
coloured using two colours with no two adjacent nodes having the same colour.
We now see that $G$ complete bipartite provides a sufficient condition
for closure for all $r$ and characterizes closure for $r=1$.

\begin{prop}[{\cite[Prop. 1.6.1]{MR1743598}}]
\label{prop:bipartite}
A graph is bipartite if, and only if, it contains no odd cycle.
 \qed
\end{prop}

\begin{lemma}
\label{lem:tricycbip}
A graph $G$ is complete bipartite if, and only if,
$G$ has no triangle and every path of length $3$ forms a cycle
(of length $4$). 
\end{lemma}
\begin{proof}
For sufficiency suppose that
$G$ has an odd cycle. Then it is either a triangle or it must
contain a path of (at least) length $3$. That $G$ is bipartite now
follows from the characterization in Proposition \ref{prop:bipartite}.


Now, if $G$ was not complete, then there exists $x,y$ in 
different partitions that 
are not adjacent. Then, consider the shortest path from $x$ to $y$: 
$x,z_1,z_2\ldots,z_k,y$. This path has length at least 4. Moreover, 
$z_{k-2}z_{k_1}z_ky$ is a path that does not form a cycle.

For necessity, we immediately see that $G$ cannot have a triangle from
Proposition \ref{prop:bipartite}. And, if $G$ has a path of length $3$,
then without loss of generality the nodes are $1,2,3,4$. Then looking at
all possible cases means that completeness implies there is a cycle,
i.e.,~we have a contradiction.
\end{proof}

\begin{cor}
Suppose that $\psnr{1}$ is closed. Then $G$ is a complete bipartite graph. 
\end{cor}
\begin{proof}
If $\psnr{1}$ is closed, then the conditions
in Lemmas \ref{lem:triangle} and \ref{lem:path} fail which by
Lemma \ref{lem:tricycbip} imply that $G$ is a complete bipartite graph. 
\end{proof}

We now use a characterization of the minimum rank PSD completion of a 
complete bipartite graph to obtain sufficiency for closure for all $r$.
\begin{theorem}
Let $G$ be a complete bipartite graph. Then $\psnr{r}$ is closed for all 
$0\leq r \leq n$.
\end{theorem}
\begin{proof}
Let $a^i$ be a sequence of partial matrices with $\PP(X^i) =
a^i\to a$ and $\rank(X_i) \leq r$. Since $G$ is complete bipartite we 
can permute the vertices of so that one partition is $\{1\ldots k\}$ 
while the other is $\{k+1\ldots n\}$. This leads to matrices of the form 
$\begin{bmatrix} ? & B\\ B^T&?\\ \end{bmatrix}$,
where $B^i\to B$ is a complete matrix with $\rank(B^i) \leq r$.
By the lower semi-continuity of the rank function, we get that 
$\rank(B)\leq r$. Let $B=PQ^T$ be a full rank decomposition of $B$ and
let
\[
X=
 \begin{bmatrix}
P\\Q
\end{bmatrix}
 \begin{bmatrix}
P\\Q
\end{bmatrix}^T.
\]
We conclude that 
\[
X\in \psnr{r}, \qquad  \PP(X)=a, \quad a_{ij}=B_{ij}, \forall ij \in E.
\]
\end{proof}
\begin{cor}
\label{cor:rankone}
Let $r=1$. Then $G$ is a complete bipartite graph if, and only if,
$\psnr{r}$ is closed.
\qed
\end{cor}
\begin{remark}
\label{cor:rankonenuclnorm}
Let $Z\in \Rmn$ be given rank one data matrices that are sampled at the
coordinates $ij\in \Omega$, i.e.,~$Z_\Omega$ is given data.
Then Corollary \ref{cor:rankone} indicates that the nuclear norm
(sum of the singular values)
heuristic cannot recover all instances unless all of $Z$ is sampled.  Recall, 
e.g.,~\cite{Rechtparrilofazel}, that the nuclear norm heuristic for 
rank minimization is equivalent to solving the SDP
\[
\begin{array}{clc}
\min & \frac 12 \trace Y\\
\text{s.t.} & Y=\begin{bmatrix} A & Z\cr Z^T & B \end{bmatrix}\succeq 0\\
          &   Y_E=Z_\Omega,
\end{array}
\]
where $E$ are the edges corresponding appropriately to the coordinates $\Omega$.
From the above results we can find classes of examples where closure
fails and no rank one completion can be found. This means that we have
classes of examples where the nuclear norm heuristic fails for data
matrices $Z$ with rank one.
\end{remark}

\subsection{Independent Sets}
Finding independent sets provide sufficient conditions for closure.
First we need the following.
\begin{lemma}
\label{lem:complblock}
Let $Y_C := \begin{bmatrix} A & B\\ B^T & C\\ \end{bmatrix}\succeq 0$
be given. Then the minimum rank PSD completion problem 
\[
\begin{array}{ccc}
\min_X & \rank \left(
 \begin{bmatrix} A & B\\ B^T & X\\ \end{bmatrix} \right) \\
 \text{s.t.} & 
 Y_X=\begin{bmatrix} A & B\\ B^T & X\\ \end{bmatrix}\succeq 0
\end{array}
\]
has optimal solution $X^*= B^TA^{\dagger}B$, where $\cdot^\dagger$ denotes
the Moore-Penrose generalized inverse. Moreover, $\rank(Y_{X^*}) =
\rank(A)$.

Then $X$ is positive semidefinite and $\rank(X) = \rank(A)$.
$X = \begin{bmatrix} A & B\\ B^T & C\\ \end{bmatrix}\succeq 0$.
	
	where $A$ is positive definite and $C = B^TA^{-1}B$. Then $X$ is positive semidefinite and $\rank(X) = \rank(A)$.
\end{lemma}
\begin{proof}
From the full rank factorization using the unique PSD square roots,
we have
\[
Y_C= \begin{bmatrix} A^{1/2} \cr C^{1/2} \end{bmatrix}
     \begin{bmatrix} A^{1/2} \cr C^{1/2} \end{bmatrix}^T, \qquad
B =  A^{1/2} C^{1/2}.
\] 
This means that $\range(B) \subseteq \range(A)$. The range condition
implies that the projection can be discarded 
$B=A^{1/2}(A^{1/2})^\dagger B$. We now define
\[
\begin{array}{rcl}
 Y_{X^*}
    &=&
\begin{bmatrix} A^{1/2} \cr B^T(A^{1/2})^\dagger \end{bmatrix}
     \begin{bmatrix} A^{1/2} \cr B^T(A^{1/2})^\dagger\end{bmatrix}^T
    \\&=&
      \begin{bmatrix} A  & 
 A^{1/2}(A^{1/2})^\dagger B \cr 
 B^T(A^{1/2})^\dagger A^{1/2} &
         B^T(A^{1/2})^\dagger 
         (A^{1/2})^\dagger B
\end{bmatrix}
    \\&=&
      \begin{bmatrix} A  & B \cr 
 B^T& B^TA^\dagger B
\end{bmatrix}.
\end{array}
\]
\end{proof}

Recall that an \textdef{independent set} (or stable set)
is a set vertices in a graph, no two of which are adjacent.
We can always complete the corresponding matrix to rank at most $n-k$.

\begin{cor}
Let $G$ be a graph with an independent set of size $k$. Then 
$\psnr{(n-k)}$ is closed.
\end{cor}
\begin{proof}
An independent set of size $k$ means that we can apply Lemma
\ref{lem:complblock} with the free block $C$ of size $k$. 
More precisely, we can pick the diagonal elements of the $A$ and $C$
blocks large enough so that $Y_C\succeq 0$ exists. Then we can always find a
completion with rank at most $\rank(A)\leq n-k$.
\end{proof}

Of course, determining if a graph has an independent set of certain size
is generally a hard problem. However some nice corollaries follow.
The following was already given in Corollary \ref{cor:nm1}.
\begin{cor}
$\psnr{(n-1)}$ is closed.
\end{cor}
\begin{proof}
Every vertex is by itself is an independent set.
\end{proof}
\begin{cor}
$\psnr{(n-2)}$ is not closed if and only if $G$ is the complete graph $K_n$.
\index{complete graph, $K_n$}
\index{$K_n$, complete graph}
\end{cor}
\begin{proof}
The only graph without an independent set of size at least two is the
complete graph $K_n$.
\end{proof}

\subsection{Cliques}

\begin{lemma}
\label{lem:cliqueclosed}
Let $G$ have a clique of size $k>2$. Then $\psnr{(j-2)}, \, j=3,\ldots,k$,
is not closed.
\end{lemma}

\begin{proof}
Let $\{1,2,\ldots k\}$ be the clique. Let $x^i\in \R^n$ be a 
sequence of vectors defined by:

\beqs
x^i_j = \left\{
	\begin{array}{lcl}
		\frac{1}{i}, & \text{if } j<k\\
		i, & \text{if } j=k\\
		0, & \text{if } j>k.
	\end{array}
\right.
\eeqs
Define the rank one sequence of PSD matrices 
\[
X^i = x^i{x^i}^T= \begin{bmatrix}
\frac 1{i^2}J & e & 0\cr
e^T           & i^2 & 0\cr
0              & 0 & 0
\end{bmatrix}
\to \begin{bmatrix}
0 & e & 0\cr
e^T           & \infty & 0\cr
0              & 0 & 0
\end{bmatrix},
\]
where $e$ is the vector of ones and $J$ is the matrix of ones.
Then $a^i := \PP(X^i) \in \psnr{1}$. But, as $a^i \rightarrow \bar{a}$ 
\index{$e$, vector of ones}
\index{vector of ones, $e$}
\index{$J$, matrix of ones}
\index{matrix of ones, $J$}
we have that $\bar{a}$ has a $(k-1) \times (k-1)$ submatrix with all $0$ 
off-diagonal entries, and free diagonal entries. 
Moreover the $(j,k)$ and $(k,j)$ entries are one, non-zero, for $j<k$.
Since the diagonal is free, we see that there is a rank $k-1$ completion
but no smaller rank completion, i.e.,~$\bar{a} \notin \psnr{(k-2)}$.
\end{proof}
\begin{theorem}
Let $G$ have disjoint cliques $C_i$ with cardinalities $k_i$ and
integers $j_i$ satisfying $|C_i|=k_i\geq
j_i>2, \, i  =1,\ldots,t$. Let $\II \subseteq \{1,\ldots,t\}$, and
$t= \sum_{i\in \II}(j_i-2)$.  Then
\[
\psnr{t} \text{   is not closed}.
\]
\end{theorem}
\begin{proof}
From Lemma \ref{lem:cliqueclosed} we have that a completion can lose
rank $j_i-2$ for each clique.
\end{proof}

\section{Conclusion}
\label{sect:concl}
In this paper we have studied the problem when the rank restricted
coordinate shadows $\psnr{r}$ are closed.

\subsection{Summary of Closure Conditions}
\begin{enumerate}
\item 
$\psnr{0}$ is trivially closed.
\item 
$\PP(\Snp)$ is closed if, and only if, 
$\LL$ and $\LL^c$ are \underline{disconnected}.
\item 
$\LL$ and $\LL^c$ are \underline{connected} implies
$\psnr{r}$ is \underline{not} closed for all $r$.
\item
$\PP(\Snrp)$ is closed if, and only if, the restricted
projections $\PP_{H_i}(\Snirpc)$ are closed for all
\underline{connected} components
$H_i$ of $G$, $n_i=|H_i|$.
\item We can now assume that $\LL$ and $\LL^c$ are \underline{not connected}.
\begin{enumerate}
\item
$\psnr{r}$ \underline{is} closed for all $r$ such that
\[
\min \left\{n-1, \left\lceil {-\frac 32 + 
            \frac {\sqrt {9 + 8|E|}}2} \right\rceil\right\}
\leq r \leq n.
\]
\item
$|\LL|=n$, a loop graph, implies that
$\psnr{r}$ \underline{is}  closed for all $r$.
\item
We can now assume that $G$ is connected and $|\LL|=0$, a loopless graph.
\begin{enumerate}
\item
for complete bipartite:
\begin{enumerate}
\item
If $G$ is complete bipartite, then $\psnr{r}$ \underline{is} closed for all $r$.
\item
$G$ is complete bipartite if, and only if, $\psnr{1}$ \underline{is} closed.
\end{enumerate}
\item
for independent set:
\begin{enumerate}
\item
If $G$ has an independent set of size $k$, then 
$\psnr{n-k}$ \underline{is} closed.
\item
$\psnr{(n-1)}$ \underline{is} closed.
\item
$\psnr{(n-2)}$ is \underline{not} closed if and only if $G$ is $K_n$.
\end{enumerate}
\item
for clique:
\begin{enumerate}
\item
If $G$ has a clique of size $k>2$, then $\psnr{(k-2)}$ is \underline{not} 
closed. (And extensions to more disjoint cliques.)
\end{enumerate}
\end{enumerate}

\end{enumerate}
\end{enumerate}
\subsection{Open Questions}
We saw that complete bipartite characterized closure for rank one and
was sufficient for all $r$. A reasonable conjecture is that for non-bipartite
graphs, we get that tripartite characterizes
closure for rank $2$ and is sufficient for $r\geq 3$. 
This naturally leads to the corresponding conjecture for higher ranks
and higher multipartite graphs. Note that a simple proof for sufficiency
for the tripartite case follows if the matrices $A,B,C$ in the partial
symmetric matrix
$\begin{bmatrix}
? & A & B\\
A^T & ? & C\\
B^T & C^T & ?\\
\end{bmatrix}$
are all rank $2$ and all $2\times 2$. We could then explicitly solve for
$P,Q,R$ in the three equations $A=PQ^T, B=PR^T, C=QR^T$ to obtain the
rank $2$ PSD completion
$\begin{bmatrix}
P\cr
Q \cr
R
\end{bmatrix}$
$\begin{bmatrix}
P\cr
Q \cr
R
\end{bmatrix}^T$.

In Remark \ref{cor:rankonenuclnorm} we have emphasized that there are
instances where the nuclear norm fails to recover the data matrix $Z$ of
the correct rank. This leads to questions about the measure of the sets
where failure occurs and is related to the conditions on sampling
for high probability completions, see \cite{MR2565240}.

\addcontentsline{toc}{section}{Index}
\printindex

\addcontentsline{toc}{section}{Bibliography}
\bibliography{.master,.edm,.psd,.bjorBOOK}

\def\cprime{$'$} \def\cprime{$'$} \def\cprime{$'$}
  \def\udot#1{\ifmmode\oalign{$#1$\crcr\hidewidth.\hidewidth
  }\else\oalign{#1\crcr\hidewidth.\hidewidth}\fi} \def\cprime{$'$}
  \def\cprime{$'$} \def\cprime{$'$} \def\cprime{$'$}
\begin{thebibliography}{10}

\bibitem{MR1404832}
A.~Auslender.
\newblock Closedness criteria for the image of a closed set by a linear
  operator.
\newblock {\em Numer. Funct. Anal. Optim.}, 17(5-6):503--515, 1996.

\bibitem{MR1797294}
A.~Barvinok.
\newblock A remark on the rank of positive semidefinite matrices subject to
  affine constraints.
\newblock {\em Discrete Comput. Geom.}, 25(1):23--31, 2001.

\bibitem{BorweinMoorsB:10}
J.M. Borwein and W.B. Moors.
\newblock Stability of closedness of convex cones under linear mappings ii.
\newblock {\em J. Nonlinear Anal. and Optim.}, 1:1--7, 2010.

\bibitem{MR2002b:15002}
A.A. Bostian and H.J. Woerdeman.
\newblock Unicity of minimal rank completions for tri-diagonal partial block
  matrices.
\newblock {\em Linear Algebra Appl.}, 325(1-3):23--55, 2001.

\bibitem{MR2565240}
E.J. Cand{\`e}s and B.~Recht.
\newblock Exact matrix completion via convex optimization.
\newblock {\em Found. Comput. Math.}, 9(6):717--772, 2009.

\bibitem{MR98g:52001}
M.M. Deza and M.~Laurent.
\newblock {\em Geometry of cuts and metrics}.
\newblock Springer-Verlag, Berlin, 1997.

\bibitem{MR1743598}
R.~Diestel.
\newblock {\em Graph theory}, volume 173 of {\em Graduate Texts in
  Mathematics}.
\newblock Springer-Verlag, New York, second edition, 2000.

\bibitem{DrPaWo:14}
D.~Drusvyatskiy, G.~Pataki, and H.~Wolkowicz.
\newblock Coordinate shadows of semidefinite and {E}uclidean distance matrices.
\newblock {\em SIAM J. Optim.}, 25(2):1160--1178, 2015.

\bibitem{Gpat:95}
G.~Pataki.
\newblock On the rank of extreme matrices in semidefinite programs and the
  multiplicity of optimal eigenvalues.
\newblock {\em Math. Oper. Res.}, 23(2):339--358, 1998.

\bibitem{Pataki:07}
G.~Pataki.
\newblock On the closedness of the linear image of a closed convex cone.
\newblock {\em Math. Oper. Res.}, 32(2):395--412, 2007.

\bibitem{Rechtparrilofazel}
B.~Recht, M.~Fazel, and P.~Parrilo.
\newblock Guaranteed minimum-rank solutions of linear matrix equations via
  nuclear norm minimization.
\newblock {\em SIAM Rev.}, 52(3):471--501, 2010.

\bibitem{MR90g:15039}
H.J. Woerdeman.
\newblock Minimal rank completions for block matrices.
\newblock {\em Linear Algebra Appl.}, 121:105--122, 1989.
\newblock Linear algebra and applications (Valencia, 1987).

\end{thebibliography}

\end{document}